\documentclass[11pt]{amsart}
\usepackage{amsmath,latexsym,amsfonts,amssymb,amsthm}
\usepackage{geometry}
\usepackage{fullpage}
\usepackage{mathrsfs}
\usepackage{graphicx,color}
\usepackage{tikz-cd}
\usepackage[all,cmtip]{xy}
\usepackage{enumitem}
\usepackage{verbatim}
\usepackage{bbm}
\usepackage{bm}
\usepackage{mathtools}
\usepackage[colorlinks=true, linkcolor=blue!70!black, urlcolor=purple, citecolor=blue!70!black]{hyperref}
\usepackage{tabu}
\usepackage{appendix}

\numberwithin{equation}{section}
\newtheorem{theorem}{Theorem}[section]

\newtheorem{conj}[theorem]{Conjecture}

\theoremstyle{definition}
\newtheorem{corollary}[theorem]{Corollary}
\newtheorem{definition}[theorem]{Definition}
 \newtheorem{question}[theorem]{Question}

  \newtheorem{property}[theorem]{Property}

\newtheorem{remark}[theorem]{Remark}
\newtheorem{example}[theorem]{Example}

\newtheorem{rem}[theorem]{Remark}

\newtheorem{defn-prop}[theorem]{Definition-Proposition}

\newcommand{\bC}{\mathbb{C}}

\newcommand{\bQ}{\mathbb{Q}}

\newcommand{\bG}{\mathbb{G}}

\newcommand{\calL}{\mathcal{L}}

\newcommand{\calF}{\mathcal{F}}

\newcommand{\calX}{\mathcal{X}}

\newcommand{\Hilb}{\mathrm{Hilb}}

\newcommand{\PGL}{\mathrm{PGL}}
\newcommand{\GL}{\mathrm{GL}}

\newcommand{\Proj}{\mathrm{Proj}}

\newcommand{\Aut}{\mathrm{Aut}}

\newcommand{\PP}{\mathbb{P}}
\newcommand{\bP}{\mathbb{P}}
\newcommand{\bA}{\mathbb{A}}

\newcommand{\calO}{\mathcal{O}}

\newcommand{\calM}{\mathcal{M}}

\usepackage{graphicx}
\newcommand{\bZ}{\mathbb{Z}}

\newcommand{\Supp}{\textrm{Supp}}

\newcommand{\ord}{\mathrm{ord}}

\newcommand{\vol}{\mathrm{vol}}

\newcommand{\lct}{\mathrm{lct}}

\newcommand{\tX}{{\widetilde{X}}}

\newcommand{\Fut}{\mathrm{Fut}}

\newcommand{\tDelta}{{\widetilde{\Delta}}}

\newcommand{\Ric}{\mathrm{Ric}}

\begin{document}

\title{On moduli of Fano varieties: an introduction to K-stability and K-moduli}
\author{Kristin DeVleming}
\date{\today}

\begin{abstract}
This survey article is an accompaniment to the 2025 Summer Research Institute in Algebraic Geometry Bootcamp on K-stability and K-moduli.  It is aimed at graduate students and intended to provide the necessary background to begin research on explicit K-moduli problems. 
\end{abstract}

\maketitle

\setcounter{tocdepth}{2}

\tableofcontents

\section*{Foreword}

This survey article is an introduction to K-stability and K-moduli of Fano varieties, with an emphasis on explicit computations.  My intention is that students with a first course in algebraic geometry, at the level of Hartshorne, will be able to learn the foundations of K-moduli theory by working through the text and exercises at the end of each section.  Many of the exercises are hands-on and will work through explicit examples of stable or unstable objects in K-moduli spaces.  

This is ultimately intended to be part of a larger, more comprehensive manuscript with Dori Bejleri on moduli of higher dimensional varieties. 

\section{Introduction}

We begin with a brief history of K-stability.

\begin{definition}
A smooth variety $X$ is called a \textbf{Fano} variety if $-K_X$ is ample (i.e. for some $m\gg 0$, the rational map $|-mK_X|: X \dashrightarrow \bP(H^0(-mK_X))$ is an embedding).  
\end{definition}

If $\dim X = 1$, $X$ is Fano if and only if $X = \bP^1$.  In general, $\bP^n$ is Fano for any $n$, but there are many other types of Fano varieties in higher dimensions.  When $\dim X = 2$, a Fano surface is called a \textit{del Pezzo} surface.

It is an old question in differential geometry to study when Fano varieties can be equipped with a K\"ahler-Einstein (KE) metric.  We won't be using this perspective in this article, but it provides relevant background on how K-stability came to be.    

A smooth K\"ahler variety with K\"ahler form $\omega$ is said to have a KE metric if $\omega$ satisfies the Einstein equation \[\Ric(\omega) = \lambda \omega\] for some constant $\lambda$.  A smooth projective variety with ample canonical class always admits a KE metric, proved independently by Aubin and Yau in 1978 \cite{Aubin,Yau}, and one with trivial canonical class always admits a KE metric, proved by Yau \cite{Yau}.  However, for Fano varieties, it was known much earlier that they \textit{cannot} always admit a KE metric. For example, in 1957, Matsushima proved that if $X$ is KE, then $\Aut(X)$ is reductive \cite{Matsushima}.  This condition is not even satisfied for all smooth del Pezzo surfaces.  It was of interest to differential geometers to formulate a notion for Fano varieties that precisely captured the existence of a KE metric.

Several years later, the notion of K-stability was introduced.  In 1992, Ding and Tian \cite{DingTian} introduced the generalized Futaki invariant to conjecturally capture the existence of a KE metric, and proved that the existence of such a metric implies this invariant is non-negative.  In 1997, Tian \cite{Tian} (analytically) and later Donaldson \cite{Donaldson} in 2002 (algebraically), the notion of K-stability was formally defined using the Futaki invariant, and the \textit{Yau-Tian-Donaldson Conjecture} was made: a smooth Fano variety is K-polystable if and only if it admits a KE metric.  This conjecture was proven by Chen, Donaldson, and Sun in 2012 \cite{CDS1,CDS2,CDS3} and Tian \cite{Tian15}, and has since been extended beyond the smooth case \cite{LXZ}.  We will define K-stability below, but you may be wondering:

\begin{question}
What does this have to do with algebraic geometry?
\end{question}

As we'll see shortly, the algebraic formulation of K-stability looks like other powerful notions in algebraic geometry (for example, GIT).  Also, the definition has something to do with degenerating varieties in families, so one may speculate a connection moduli problems.  However, it is \textit{remarkable} that this differential geometric notion is exactly the correct thing to study to get well-behaved moduli spaces of Fano varieties, and \textit{remarkable} that it has so many connections to older algebro-geometric concepts (e.g. singularities and the minimal model program).  

In the words of Chenyang Xu, \textit{``The concept of K-stability is one of the most precious gifts differential geometers
brought to algebraic geometers.''}

In this survey, we will define K-stability, state the K-moduli theorem and the fundamental results needed for its proof, and focus on explicit computations for Fano varieties and log Fano pairs.  This will only scratch the surface of K-stability and K-moduli in general; a comprehensive textbook on the subject has been written by Xu \cite{Xu23}.

\section*{Acknowledgments}

I was first introduced to K-stability by Yuchen Liu at an AMS Sectional Meeting as a graduate student.  I am thankful to Yuchen Liu and my other collaborators in this area for their enthusiasm and ideas: Hamid Abban, Kenneth Ascher, Dori Bejleri, Harold Blum, Ivan Cheltsov, Giovanni Inchiostro, Lena Ji, Patrick Kennedy-Hunt, Ming Hao Quek, Julie Rana, Fei Si, and Xiaowei Wang.  I also thank Chenyang Xu and Ziquan Zhuang for teaching me much about the subject.  I gratefully acknowledge support from the National Science Foundation through NSF DMS Grant 2302163. 

Finally, I am thankful to all of the participants in my 2025 SRI Bootcamp group for a productive and energetic week together: Suraj Dash, Joshua Enwright, Fernando Figueroa, Gomathy Ganapathy, Myeong Jae Jeon, Tyson Klingner, Daniel Mallory, Anda Tenie, Ying Wang, and Wern Yeong.  We have a forthcoming paper on our results from the Bootcamp.

\section{Singularities of the Minimal Model Program}\label{s:singularities}

K-stability is related to singularities of Fano varieties, and we begin with a brief introduction to the singularities of the minimal model program.  For the purposes of this survey, all varieties will be assumed normal.  A \emph{pair} $(X,\Delta)$ consists of a normal variety $X$ and an effective $\bQ$-divisor $\Delta$ on $X$ such that $\Supp(\Delta)$ does not contain a codimension 1 singular point of $X$.  We will say $(X,\Delta)$ is \textit{$\bQ$-Gorenstein} if $K_{X}+\Delta$ is $\bQ$-Cartier.  The pair $(X,\Delta)$ is  \emph{projective} if $X$ is projective.

\begin{definition}
    A log resolution of a pair $(X,\Delta)$ is a pair $(Y,\Delta_Y)$ such that $\pi: Y \to X$ is a resolution $Y$ of $X$ and, if $\{E_i\}$ are the exceptional divisors of $\pi$, $\mathrm{Supp} \ \pi^{-1}_* \Delta \cup \{E_i\}$ is simple normal crossing.
\end{definition}

\begin{definition}{\cite[Definition 2.28]{KM98}}
    If $(X,\Delta)$ is a pair and $K_X +\Delta$ is $\bQ$-Cartier, write $\Delta = \sum b_j \Delta_j$ where $\{ \Delta_j\}$ are prime divisors and $b_j \in \bQ^{\ge 0}$.  Let $\pi: Y \to X$ be a log resolution of singularities and $\Delta_Y = \sum b_j \tilde{\Delta}_j$, where $\tilde{\Delta}_j$ is the strict transform of $\Delta_j$ on $Y$.  Let $\{ E_i \}$ be the components of the exceptional locus.  Then, for some rational numbers $a_{X,\Delta}(E_i)$, we can write
    \[ K_Y + \Delta_Y = \pi^* (K_X+\Delta) + \sum_{E_i} a_{X,\Delta}(E_i) E_i .\]

    We define the \textit{discrepancy} of the pair $(X,\Delta)$ to be \[\mathrm{discrep } (X,\Delta) = \inf_{E/X: E \text{ exceptional }} a_{X,\Delta}(E).\]  Note the infimum is taken over all exceptional divisors in any log resolution of $(X,\Delta)$.

    The pair $(X,\Delta)$ is said to have:
        \begin{itemize}
            \item \textit{terminal singularities} if $\mathrm{discrep } (X,\Delta) > 0 $,
            \item \textit{canonical singularities} if $\mathrm{discrep } (X,\Delta) \ge 0 $,
            \item \textit{Kawamata log terminal singularities} (klt) if $\mathrm{discrep } (X,\Delta) > -1 $ and $b_j < 1$ for each $j$, 
            \item \textit{purely log terminal singularities} (plt) if $\mathrm{discrep } (X,\Delta) > -1 $, and
            \item \textit{log canonical singularities} (lc) if $\mathrm{discrep } (X,\Delta) \ge -1 $.
        \end{itemize}
\end{definition}

The discrepancies can be computed on one log resolution, provided $\Delta_Y$ is smooth (i.e. any intersections among the $D_j$ are separated): 

\begin{theorem}{\cite[Corollary 2.32]{KM98}}\label{t:discrepofpair}
    Let $\pi: (Y,\Delta_Y) \to (X,\Delta)$ be a log resolution such that $\Delta_Y$ is smooth, where $\Delta_Y$ is the strict transform of $\Delta$ on $Y$.  If $a_{X,\Delta}(E_i) \ge -1$ for every exceptional divisor $E_i$ of $\pi$, then $\mathrm{ discrep } (X,\Delta) = \min \{ \min_i \{a_{X, \Delta}(E_i)\} , \min_j \{ 1- b_j \}, 1 \} $.  
\end{theorem}

Given a non-necessarily log canonical pair, we also need the following: 

\begin{definition}\label{defn:lc}
    If $\Delta \subset X$ is a nonzero $\bQ$-Cartier divisor on a normal, log canonical variety $X$ with $\bQ$-Cartier canonical divisor, the \textit{log canonical threshold} of $\Delta$ is the maximum nonnegative number $c$ such that $(X,c\Delta)$ is log canonical.  This is denoted $\lct(X,\Delta)$ or simply $\lct(\Delta)$ if $X$ is clear from context. By Theorem \ref{t:discrepofpair}, we see that $c \le 1$. 
\end{definition}

\begin{example}
    Let $X$ be the cone $xy = z^2 \subset \bA^3$.  Blowing up the origin $(0,0,0) \subset \bA^3$, we obtain $\pi: Z \to \bA^3$.  Let $Y$ be the strict transform of $X$ and let $\pi_Y = \pi|_Y: Y \to X$.  This is a resolution of singularities of $X$.  Indeed, $Y$ is smooth by the local equations of the blow-up.  Furthermore, if $E \subset Z$ is the exceptional divisor of $\pi$, $E \cong \bP^2$, and $E_Y := E \cap Y$ is the exceptional divisor of $\pi_Y$.  Then, $E_Y$ is a conic in $E$ (seen from the blow up equation) and $(E_Y)^2 = E \cdot E \cdot Y = E|_E \cdot Y|_E$.  Because $\calO_E(E) = \calO_{\bP^2}(-1)$, we see that $(E_Y)^2 = -2$.   

    Therefore, we may compute $a_X(E_Y)$: consider the equality \[ K_Y = \pi_Y^*K_X + a_X(E_Y)E_Y.\]

    Adding $E_Y$ to each side, we obtain 
    \[ K_Y+E_Y = \pi_Y^*K_X + (a_X(E_Y)+1)E_Y.\]

    Taking the intersection with $E_Y$, we find 
    \begin{align*}
         (K_Y+E_Y)\cdot E_Y &= (\pi_Y^*K_X + (a_X(E_Y)+1)E_Y)\cdot E_Y \\
         -2 &= 0 + (a_X(E_Y)+1)(-2)
    \end{align*}
    where the left side holds by adjunction and $\pi_Y^*K_X \cdot E_Y = 0$ by the projection formula.  Therefore, $a_X(E_Y) + 1 = 1$, so $a_X(E_Y) = 0$.  In particular, we have found $\mathrm{ discrep }(X) = 0$ and hence $X$ is canonical.
\end{example}

\begin{example}
    Let $X$ be a cone over an elliptic curve.  By blowing up the cone point, we obtain a resolution $\pi: Y \to X$ with exceptional divisor $E$ an elliptic curve.  We compute 
    \begin{align*}
        K_Y &= \pi^*K_X + a_X(E) E \\
        K_Y +E &= \pi^*K_X + (a_X(E)+1) E \\
        (K_Y +E)\cdot E &= (\pi^*K_X + (a_X(E)+1) E)\cdot E \\
        0 &= 0 + (a_X(E)+1)E^2
    \end{align*}

    Because $E$ is contractible, we have $E^2 < 0$, so this implies $a_X(E) = -1$ and hence $\mathrm{ discrep }(X) = -1$ so $X$ is log canonical.
\end{example}

\subsection{Exercises}

\begin{enumerate}

    \item Prove that terminal surface singularities are smooth.
    
    \item Compute the log canonical thresholds of the pair $(\bA^2, C)$, where $C$ is: 
        \begin{enumerate}
            \item $y^2 = x^3$
            \item $y^2 = x^4$
            \item $y^2 = x^n$, $n \ge 3$
            \item $n$ lines passing through the origin
        \end{enumerate}
        
    \item Let $X$ be the cone over the rational normal curve of degree $d$.  Compute $\mathrm{ discrep }(X)$. 

    \item Let $X$ be the cone over the rational normal curve of degree $d$ and let $D$ be a ruling of the cone.  Compute $\mathrm{ discrep } (X,D)$.

    \item Let $X$ be a cone over a curve of genus $g$ with $g > 1$.  Prove that $X$ is not log canonical.

    \item Give an example of a variety $X$ with non $\bQ$-Cartier canonical divisor.
\end{enumerate}

\section{Introduction to K-stability}\label{s:introtoKstability}

Our first goal is to introduce the notion of K-stability with connections to other invariants and classify K-(semi)stable smooth del Pezzo surfaces.

\subsection{K-stability via test configurations} 

Without further ado, let's define K-stability.  We will not restrict ourselves to the smooth world; let us consider arbitrary normal projective varieties.  

\begin{definition}[Tian \cite{Tian}, Donaldson \cite{Donaldson}]
Let $(X,L)$ be a polarized projective variety of dimension $n$, and suppose $X$ is normal.  Because $L$ is ample, for $m \gg 0$, there is an embedding $|L^m|: X \to \bP^N$.  For any action $\bG_m$ on $\bP^N$, there in an induced action $\bG_m$ on the class $[X] \in \Hilb(\bP^N)$.  Let $[X_0] = \lim_{t\to 0} t \cdot [X]$. 

A \textit{test configuration} is the induced family 

\begin{center}
\begin{tikzcd}
\bG_m \times (X, L^m) \arrow[r] \arrow[d] & (\calX, \calL) \arrow[d] \\
\bG_m = \mathbb{A}^1 \setminus \{0\} \arrow[r] & \mathbb{A}^1 \\
\end{tikzcd}
\end{center}
\end{definition}

\begin{remark}
    If you have seen GIT, this definition should look somewhat familiar to 1-parameter subgroups in GIT.  Indeed, any 1-parameter subgroup will induce a test configuration.  The key difference between the two is that test configurations are not restricted to a fixed ambient projective space: we can continue to take higher powers of the line bundle $L$ to re-embed $X$ into larger projective spaces.
\end{remark}

Given a test configuration, by Riemann-Roch, we can compute 

\begin{align*}
    d_k := h^0 (X, L^k) &= a_0 k^n + a_1 k^{n-1} + \dots \\
    &= \frac{L^n}{n!} k^n - \frac{L^{n-1} \cdot K_X}{2(n-1)!} k^{n-1} + \dots .
\end{align*}

Since $\bG_m$ is acting on $(\calX, \calL)$, is it acting on $(\calX_0, \calL_0) =: (X_0, L_0)$, and hence $H^0(X_0, L_0)$.  For $k \gg 0$, the total weight of this action on $H^0(X_0, L_0^{k/m})$ also grows as a polynomial: 

\[ w_k = b_0 k^{n+1} + b_1 k^n + \dots .\]

 \begin{rem}
    To compute the $b_i$ and show that this is a polynomial, we can complete the family $(\calX, \calL)$ over $\mathbb{A}^1$ to a family $(\overline{\calX}, \overline{\calL})$ over $\bP^1$ by adding the trivial fiber $(X,L^m)$ over $\infty \in \bP^1$.  We do this by gluing the family $(\calX, \calL)$ to the trivial family $X \times \bP^1 \setminus \infty$ along $\mathbb{A}^1 \setminus 0$. 
 Then, we have a test configuration $(\overline{\calX}, \overline{\calL}) \to \bP^1$ with a $\bG_m$ action and can use equivariant Riemann-Roch to compute the weight:

\begin{align*}
    w_k &= b_0 k^{n+1} + b_1 k^n + \dots \\
    &= \frac{\overline{\calL}^{(n+1)/m}}{(n+1)!}k^{n+1} - \frac{\overline{\calL}^{(n/m)} \cdot K_{\overline{\calX}/\bP^1}}{2n!} k^n + \dots . 
\end{align*} 

For related discussion, see \cite[\S 2.1.2]{Xu23}.
 \end{rem}

\begin{definition}[Tian, Donaldson]
The \textit{generalized Futaki invariant} $\Fut(\calX, \calL)$ of the test configuration $(\calX, \calL)$ is \[ \Fut (\calX, \calL) = \frac{a_1 b_0 - a_0 b_1}{a_0^2}.\]
\end{definition}

This expression comes from the quotient \[ F(k) := \frac{w_k}{kd_k} = F_0 + F_1 \frac{1}{k} + F_2 \frac{1}{k^2} + \dots \] and \[\Fut(\calX, \calL) = - F_1.\]  While not immediately obvious form the definition, this invariant is very closely related to Hilbert stability and Chow stability (which are defined in similar ways, by Mumford).  

In practice, for $X$ Fano, we use $L = - mK_X$.  We will exclusively use this in what follows, and $\Fut(\calX, \calL)$ takes on a particularly nice form \cite[Prop. 2.17]{Xu23}: 

\[ \Fut(\calX, \calL) = \frac{1}{2(-K_X)^n} \left( \left( \frac{1}{m} \overline{\calL} \right)^n \cdot K_{\overline{\calX}/\bP^1} + \frac{n}{n+1} \left(\frac{1}{m} \overline{\calL} \right)^{n+1} \right) .\] 

We will use the Futaki invariant to define K-stability, but make one further simplification.  It turns out that we can restrict the test configuration definition to only sufficiently``nice'' varieties $X_0$.  

\begin{definition}
A test configuration $(\calX, \calL)$ is called a \textbf{special test configuration} if $\calX$ is a $\mathbb{Q}$-Gorenstein family of $\mathbb{Q}$-Fano varieties, i.e. $\calL \sim -mK_{\calX}$ and $X_0$ has klt singularities.  
\end{definition}

\begin{theorem}[\cite{LX14}]\label{thm:specialtestconfig}
To test K-(semi/poly)stability, one only needs to test special test configurations. 
\end{theorem}

We give a very brief idea of the proof; for more, see \cite{LX14}.

\begin{proof}[Idea of proof.]
Use the Minimal Model Program (MMP).  Starting with any test configuration $(\calX, \calL)$, perform birational modifications (MMP operations) or finite base change and normalization to produce a special test configuration $(\calX^s, -mK_{\calX^s})$.  Then, show that the Futaki invariant can only decrease under these birational operations.  
\end{proof}

\begin{remark}
If $(\calX, \calL)$ is a special test configuration, \[ \Fut(\calX, \calL) = - \frac{1}{2(-K_X)^n(n+1)} \left(-K_{\overline{\calX}/\bP^1} \right)^{n+1} .\]
In particular, for special test configurations, the sign of the Futaki invariant is determined by the sign of $\left(-K_{\overline{\calX}/\bP^1} \right)^{n+1}$.
\end{remark}

With this definition, we define K-semistability.

\begin{definition}[Tian, Donaldson]\label{def:stability}
Let $X$ be a variety such that $-K_X$ is ample.  $X$ is 

\begin{enumerate}
    \item K-semistable if $\Fut(\calX, \calL) \ge 0$ for all special test configurations $(\calX, \calL)$.
    \item K-stable if $\Fut(\calX, \calL) \ge 0$ for all special test configurations $(\calX, \calL)$, and equality holds if and only if $(\calX, \calL)$ is trivial ($\calX \cong X \times \mathbb{A}^1$ and the $\bG_m$-action on fibers is trivial).
    \item K-polystable if $X$ is K-semistable and, if $\Fut(\calX, \calL)= 0$ for any special test configuration, then $\calX \cong X \times \mathbb{A}^1$ is a product.
\end{enumerate}
\end{definition}

\begin{rem}
    By \cite{BX19}, if $X$ is K-stable, then $\Aut(X)$ is finite, and if $\Aut(X)$ is finite, then $X$ is K-stable if and only if it is K-polystable.  We will explore K-polystability in more detail in later sections.  If you are familiar with GIT, think of `polystable' as a closed orbit condition as it is in GIT.  
\end{rem}

If the complexity of the definition did not scare you off, hopefully here you are seeing a \textcolor{red}{RED FLAG}.  This definition depends on the $\bG_m$ action and the power $m$ used for $L$! 

\begin{remark}[Red Flag!]\label{rmk:checkable}
To test if a variety is K-(semi/poly)stable, we must a priori test \textit{infinitely} many test configurations, which depend on the $\bG_m$ action \textit{and} the power $m$ used in the embedding $|L^m|: X \to \bP^N$.  (For those familiar with GIT: this is like checking the Hilbert-Mumford weight for every possible embedding of $X$ into a higher and higher projective space.)  How can this possibly be reasonable?  
\end{remark}

We first make connections with singularities in algebraic geometry.  Although nothing about singularities explicitly appears in the definition of the Futaki invariant, asking that a variety is K-semistable has (surprising!) consequences on the singularities of $X$.  

\begin{theorem}[\cite{Odaka13}]\label{thm:klt}
If $-K_X$ is ample, then K-semistability of $X$ implies that $X$ has log terminal singularities.  
\end{theorem}

\subsection{K-stability via the $\alpha$-invariant}

In the rest of this section, we will primarily focus on \textit{other invariants} that capture K-(semi)stability.  These can be easier to check in practice. 

We first define the $\alpha$-invariant, introduced by Tian in \cite{Tian87}.  The original definition is analytic, but by Theorem A.3 of Demailly's appendix in \cite{CS08}, it coincides with what follows.

\begin{definition}[Tian]
Let $X$ be a $\mathbb{Q}$-Fano variety.  Tian's $\alpha$-invariant is 
\[ \alpha(X) = \inf_{0 \le D \sim_{\mathbb{Q}} - K_X} \lct(X,D) .\]
\end{definition}

\begin{example}
If $X = \bP^n$, because $-K_{\bP^n} = (n+1) H$, \[ \alpha(\bP^n) =\frac{1}{n+1}.\]
\end{example}

\begin{theorem}[\cite{Tian87}]\label{thm:alphainvariant}
Let $X$ be a $\mathbb{Q}$-Fano variety of dimension $n$.  If \[ \alpha(X) > (\ge) \frac{n}{n+1} \]
then $X$ is K-(semi) stable.  
\end{theorem}

This is not an if-and-only-if, but can be readily computable, and we can use the $\alpha$ invariant to check K-stability of Fano varieties.

\begin{example}
    If $X = \bP^1$, $\alpha(X) = \frac{1}{2}$ so $\bP^1$ is K-semistable.  But, for $n > 1$, the $\alpha$-invariant tells us nothing.  We will see later that a refinement of this criterion can be used to show that $\bP^n$ is always K-semistable.
\end{example}

Let's use this to understand the stability of some del Pezzo surfaces.  

\begin{definition}
The \textbf{degree} of a del Pezzo surface $X$ is $d = (-K_X)^2$.  For $X = \bP^2$, $ d = 9$.  For $X = \bP^1 \times \bP^1$, $d = 8$.
\end{definition}

\begin{example}
Let $X$ be a del Pezzo surface of degree 1 (so $X$ is the blow up of $\bP^2$ at 8 points).  We will directly compute $\alpha(X)$ and show that $X$ is K-stable.  

Consider the linear system $|-K_X|$.  In terms of curves on $\bP^2$, 
\[ -K_X = \pi^*(-K_{\bP^2}) - \sum_{i = 1}^8 E_i \] so the curves in $|-K_X|$ are (strict transforms of) cubic curves in $\bP^2$ that pass through all 8 points that were blown up.  Suppose that $D \sim_\bQ - K_X$.  If $\Supp D \notin |-K_X|$, pick $x \in D$ and choose a curve $C \in |-K_X|$ such that $x \in C$.  Because $D \sim_\bQ -K_X$ and $C \sim_\bQ - K_X$, $D \cdot C = (-K_X)^2 = 1$, so the multiplicity of $D$ at every point is at most $1$.  This implies that $(X,D)$ is log canonical, so $\lct(X,D) = 1$. 

Then, to finish computing $\alpha(X)$, we just need to compute the log canonical threshold of curves $D \in |-K_X|$.  Each such $D$ is a cubic plane plane curve vanishing at the 8 points we blew up (so, in particular, $D$ must be reduced and irreducible, because no three of the points blown up were co-linear, and no 6 were on a conic), and such curves are either: 
    \begin{align*}
        &\text{smooth} \\
        &\text{nodal} \\
        &\text{cuspidal}
    \end{align*}
and the log canonical threshold of a pair $(X,D)$ with $D$ as above is: 
    \begin{align*}
        &\text{smooth}: \lct(X,D) = 1 \\
        &\text{nodal}: \lct(X,D) = 1 \\
        &\text{cuspidal}: \lct(X,D) = \tfrac{5}{6}. 
    \end{align*}
Therefore, in each case, we see that $\alpha(X) = \min \lct(X,D) \ge \frac{5}{6} > \frac{2}{3}$ so by Theorem \ref{thm:alphainvariant}, $X$ is K-stable.  
\end{example}

\begin{remark}
More generally, Cheltsov \cite{Chel08} has shown that $\alpha(X) \ge \frac{2}{3}$ for $X$ a del Pezzo surface of degree $\le 4$, hence any such $X$ is K-semistable. 
\end{remark}

There are many refinements of Tian's criteria.  None of these are if-and-only-ifs, but at least they make it possible to check K-stability in many cases. 

\begin{theorem}[\cite{Fuj19a}]\label{thm:fujitarefinement}
If $X$ is a surface, or smooth of dimension $\ge 3$, and 
\[ \alpha(X) \ge \frac{n}{n+1}, \] then $X$ is K-stable. 
\end{theorem}

\begin{corollary}\label{cor:deg4}
All del Pezzo surfaces of degree $\le 4$ are K-stable. 
\end{corollary}

\begin{example}
If $X$ is a smooth hypersurface in $\bP^{n+1}$ of degree $n+1$, then Cheltsov and Park \cite{CP02} have shown $\alpha(X) \ge \frac{n}{n+1}$.  Therefore, by Theorem \ref{thm:fujitarefinement}, all such $X$ are K-stable.  
\end{example}

What about $\bP^n$?  Surely we should be able to determine if it is semistable or not.  We use a $G$-invariant version of the $\alpha$-invariant. 

\begin{definition}
Let $X$ be a Fano variety with a group action by an algebraic group $G$.  Define
\[ \alpha_G(X) = \inf_{0 \le D \sim_{\mathbb{Q}} - K_X, D\text{ is } G\text{-invariant }} \lct(X,D) .\]
\end{definition}

Proved in increasing levels of generality by \cite{DS,LX,LZ,Zhuangequivariant}, we have the following.

\begin{theorem}
Let $X$ be a Fano variety with a group action by an algebraic group $G$.  
    \begin{enumerate}
        \item If \[ \alpha_G(X) \ge \frac{n}{n+1} ,\]
then $X$ is K-semistable. 
        \item If $G$ is reductive and \[ \alpha_G(X) > \frac{n}{n+1}, \]
then $X$ is K-polystable. 
    \end{enumerate}
\end{theorem}

\begin{example}
Let $X = \bP^n$ and $G = \mathrm{PGL}(n+1)$.  There are no $G$-invariant divisors, hence $\alpha_G(X) = \infty$, so $\bP^n$ is K-polystable. 
\end{example}

\subsection{K-stability via the $\beta$ and $\delta$ invariants}

Thus far, we have introduced the test configuration definition for K-stability and the $\alpha$-invariant (and $\alpha_G$) which we could use as a test to determine K-stability.  The $\alpha$-invariant had the advantage that it is relatively computable, however, it has a distinct disadvantage of not being an if-and-only-if statement.  Can we get such a statement?  Because Theorem \ref{thm:klt} ties K-semistability to the singularities and birational geometry of $X$, we may hope for a definition of K-(semi/poly)stability in more birational geometric terms. In fact, we can connect the Futaki invariant to Fujita and Li's $\beta$-invariant (or $\delta$ invariant) .  For the reader's convenience, note that in \cite{Xu23}, the $\beta$-invariant is called the Fujita-Li invariant. 

\begin{definition}
Let $X$ be a $\mathbb{Q}$-Fano variety and $E$ a prime divisor over $X$.  Let $\mu: Y \to X$ be any morphism such that $E \subset Y$.

Let $A_X(E)$ be the log discrepancy of the divisor $E$, defined as \[ A_X(E) = 1 + \ord_E(K_Z - f^*K_X) = 1 + a_X(E).\]

Define $S_X(E)$ to be 
\[ S_X(E) = \frac{1}{(-K_X)^n} \int_0^\infty \vol(\mu^*(-K_X) - tE) dt .\]  This does not depend on choice of $\mu$ and $Y$, so we often write \[ S_X(E) = \frac{1}{(-K_X)^n} \int_0^\infty \vol(-K_X- tE) dt .\]

The $\beta$-invariant of the divisor $E$ is 
\[ \beta_X(E) = A_X(E) - S_X(E) .\]  The $\delta$-invariant of $E$ is \[\delta_X{E} = \frac{A_X(E)}{S_X(E)}. \]
\end{definition}

In the integral $S$, we must compute $\vol(D)$ for $D = -K_X - tE$.  What follows are some notions related to volumes of divisors. 

\begin{definition}
    The \textit{volume} a divisor $D$ on a normal variety $X$ of dimension $n$ is 
    \[ \vol(D) = \lim_{m\to \infty} \frac{h^0(X,mD)}{m^n/n!}.\]
\end{definition}

\begin{definition}
A divisor $D$ on a normal variety $X$ of dimension $n$ is \textit{big} if one of the following equivalent definitions hold: 
    \begin{enumerate}
        \item For $m \gg 0$, the map given by the linear system $|mD|: X \dashrightarrow \bP^N$ is birational onto its image.
        \item For $m \gg 0$, there exists a constant $c>0$ such that $h^0(X,mD)>cm^n$.
        \item $\vol(D) > 0$.
    \end{enumerate}
\end{definition}

How do we compute volumes?  If $D$ is a divisor on a normal variety $X$ of dimension $n$, we offer the following computational remarks: 

\begin{enumerate}
    \item If $D$ is nef, $\vol(D) = D^n$. 
    
    \item If $D$ is big, by definition $\vol(D) > 0$, and we can at least bound the volume of $D$ from below by considering the image of the linear system $|mD|: X \dashrightarrow \bP^N$.  Because $D$ is big, the image is a variety $Y$ birational to $X$.  If we assume this a morphism $f: X \to Y$, some divisors in $X$ may be contracted.  If we write $D = f^*\calO(1) + N$ for some effective divisor $N$ supported on the contracted locus, then $\vol(f^*(\calO(1)) \le \vol(D)$ (because $h^0(f^*\calO(1)) \subset h^0(D)$).  And, $f^*\calO(1)$ is nef, so its volume is just $\calO(1)^n$.  Therefore, we may compute a lower bound for the volume of $D$ by the volume of $\calO(1)$.  In practice, we determine $N$ by considering MMP-like birational modifications of $X$.  If the divisor $D$ is trivial or negative on some effective curve in $X$, then we contract the class of this curve.  If this contraction is divisorial, $N$ will be supported on the divisor, and if it is a small contraction, we flip or flop the class of the curve to a new variety $X \dashrightarrow X^+$ and pullback everything to a common partial resolution $\widehat{X}$.  The divisor $N$ will be supported on the exceptional divisors of $\widehat{X} \to X$.

    For surfaces, this is called a \textit{Zariski decomposition}: we can always write a big divisor $D$ on a surface $X$ as $D = P + N$, where $P$ is nef, $N$ is negative (meaning it has negative definite intersection matrix and is contractible, or is $0$), and $P\cdot N = 0$.  In this case, we have the equality $\vol(D) = \vol(P) = P^2$.  To relate this to the birational modifications above, because $N$ is negative, it corresponds to a contractible curve in the surface, and $P$ is the pushforward of the divisor $D$ on the contraction.
\end{enumerate}

We will see some examples of computing volumes in the exercises.

\begin{remark}
In the definition of $S_X(E)$, we need to compute an improper integral.  However, $\vol(-\mu^*K_X - tE) > 0$ if and only if $-\mu^*K_X - tE$ is big.  In terms of divisors on $Y$, the closure of the big cone of divisors is the pseudo-effective cone, so the volume is only non-zero if $t \in [0, \tau ]$ where $\tau$ is the pseudo-effective threshold.  This is finite; at some point we have subtracted `too much' $E$ and the divisor is no longer pseudo-effective.  Therefore, we could re-write 
\[ S_X(E) = \frac{1}{(-K_X)^n} \int_0^\tau \vol(-K_X- tE) dt .\]
\end{remark}

With this definition, we can relate the K-(semi/poly)stability of $X$ intrinsically to the birational geometry of $X$.  The following theorem is called the \textbf{valuative criterion} for K-(semi/poly)stability.  Initially given by Fujita and Li, there are several important contributions necessary connecting uniform K-stability and K-stability by the other cited authors.

\begin{theorem}[\cite{FuValuative,LiValuative,FO16, BJ17,LXZ}]\label{thm:valcriteria}
A variety $X$ is K-semistable (resp. stable) if and only if $\beta_X(E) \ge 0$ (resp. $> 0$) for all prime divisors $E$ over $X$ (equivalently, $\delta_X(E) \ge 1$ (resp. $> 1$)).
\end{theorem}

This should be taken with a (possibly smaller) \textcolor{red}{\small{RED FLAG}} than the test configuration definition, because while it may appear more birational geometric in nature, it still requires checking every divisor $E$ over $X$.  However, it is a nice complement to the $\alpha$-invariant criteria because it allows us to prove that certain varieties are \textit{not} (semi)stable.  Such examples can be found in the exercises.  We provide a sample computation below. 

\begin{example} 
Let's compute $\beta_{\bP^2}(E)$ where $E$ is the exceptional divisor of a blow up of a point on $\bP^2$.  Let $\mu: Y \to \bP^2$ be the blow up.  Because $K_Y = \mu^*(K_{\bP^2}) + E$, we have $A_{\bP^2}(E) = 1+1 = 2$. 

Now we need to compute $S_{\bP^2}(E)$.  We know $(-K_{\bP^2})^2 = 9$. To compute $\vol(-K_{\bP^2} - tE)$, we need to understand when the divisor is ample, big, and nef.  The volume is non-zero when the divisor is big (by definition of being big).

We know that $\mu^*(-K_{\bP^2}) - tE = -K_Y + (1-t)E$.  The Mori cone of $Y$ is generated by $E$ and the class of a fiber $F$ of the ruled surface $Y \to \bP^1$.  Because $(-K_Y + (1-t)E)\cdot E = 1-(1-t) = t$, this is positive on $E$ for $t >0$.  Similarly, $(-K_Y + (1-t)E) \cdot F = 2 +(1-t) = 3 - t$, so this is positive on $F$ for $t < 3$.  Because this is ample exactly when it has positive intersection with both $F$ and $E$, this is ample for $0 < t < 3$.  Also, when $t = 3$, this is trivial on $F$, and the morphism induced by the linear system $|m(-K_Y+(1-t)E)|$ therefore contracts $F$ and hence contracts $Y$ to a curve.  Because this is not birational for any $m > 0$, the divisor is not big for any $t \ge 3$.  

This implies that: 
\begin{align*}
    \text{for } 0 \le t \le 3,& \quad \vol(\mu^*(-K_{\bP^2}) - tE) = (\mu^*(-K_{\bP^2}) -tE)^2 = 9-t^2 \\
    \text{for } t \ge 3, \quad & \vol(\mu^*(-K_{\bP^2}) - tE) = 0.\\
\end{align*}

So, we can compute $S_{\bP^2}(E)$: 
\[ S_{\bP^2}(E) = \frac{1}{9} \int_0^3 (9-t^2) dt = \frac{18}{9} = 2.\]

Finally, we can conclude that $\beta(E) = A(E) - S(E) = 0$.
\end{example} 

Note that this alone does not imply that $\bP^2$ is K-semistable; to use the valuative criterion, we need to show that $\beta(E) \ge 0$ for \textit{every} divisor $E$ over $\bP^2$.

\begin{remark}
    The tools in this section are sufficient to characterize the stability of del Pezzo surfaces.  
\end{remark}

We list the stability of each del Pezzo surface, along with a reason. Recall that, for a del Pezzo surface $X$, its degree is $d = (-K_X)^2$.

\begin{center}
\begin{tabular}{c|c|c}
    degree & stability & reason \\\hline
    9 & polystable & equivariant $\alpha$-invariant \\
    8 $(X = \bP^1 \times \bP^1)$ & polystable & equivariant $\alpha$-invariant, Exercise 5 \\
    8 $(X = \mathbb{F}_1)$ & unstable & $\beta$-invariant computation, Exercise 7 \\
    7 & unstable & $\beta$-invariant computation, Exercise 7  \\
    6 & polystable &  equivariant $\alpha$-invariant, Exercise 5 \\
    5 & stable & equivariant $\alpha$-invariant plus finite $\Aut(X)$, \cite{Chel08} \\
    $\le 4$ & stable & $\alpha$-invariant, Cor. \ref{cor:deg4} \\
\end{tabular}
\end{center}

\subsection{Exercises}

\begin{enumerate}

\item (a) If $X$ is the blow up of $\bP^2$ at $k$ sufficiently general points, $0 \le k \le 8$, or $X = \bP^1 \times \bP^1$, prove that $X$ is Fano. 

(b) Determine what `sufficiently general' means in the previous exercise (it should be a condition on the points that were blown up).

(c) (Harder) Prove that every smooth Fano surface is one of those listed in (a).  Such surfaces are called \textit{del Pezzo} surfaces.  

\item Give an example of a smooth Fano variety whose automorphism group is non-reductive.  By Matsushima's result \cite{Matsushima}, this cannot be K-polystable.

\item If $X$ is the blow up of $\bP^2$ at $r$ general points, $0 \le r \le 8$, prove that the degree of $X$ is $9 - r$. 

\item Show that $\bP^1 \times \bP^1$ and the blow up of $\bP^2$ at three points are K-polystable (hint: $\Aut(X)$-invariant divisors?). 

\item Show that the blow up of $\bP^2$ at a point is K-unstable by computing $\beta_X(E)$, where $E \subset X$ is the exceptional divisor of the blow up. 

\item Show that the blow up of $\bP^2$ at a point (a del Pezzo surface of degree $8$) is K-unstable by computing $\beta_X(E)$, where $E \subset X$ is the exceptional divisor of the blow up. 

\item Show that the del Pezzo surface of degree $7$ is K-unstable.

\item By blowing up the cone point and computing $\beta_X(E)$, show that the weighted projective space $\bP(1,1,n)$ is K-unstable for any $n >1$.  

\item Explain why the blow up of $\bP^3$ along a planar cubic is Fano, and show that it is K-unstable.   (Hint: there are two natural divisors to check $\beta$ of; the exceptional of the blow-up and the strict transform of the plane containing the curve.  Only one will work to show you it is K-unstable.)

\item This exercise will introduce some ideas related to K-stability of pairs.  Let $X = \bP(1,1,2)$.
    \begin{enumerate}
        \item Show that the sections of $\mathcal{O}_{X}(2)$ define an embedding $X \hookrightarrow \bP^3$ whose image is $(xy = z^2)$.  This shows directly that $\bP(1,1,2)$ is the singular quadric cone. 
        \item Show that $X$ is K-unstable.
        \item Given a log Fano pair $(X,D)$ of dimension $n$, we can define K-stability of the pair.  For any prime divisor $E$ over $X$, define $\beta_{(X,D)}(E) = A_{(X,D)}(E) - S_{(X,D)}(E)$, where $A_{(X,D)}(E)$ is the log discrepancy of the pair, and for any morphism $\mu: Y \to X$ extracting $E$, \[ S_{(X,D)}(E) = \frac{1}{(-K_X-D)^n} \int_0^\infty \mathrm{vol}(\mu^*(-K_X-D) - tE) dt .\]
        Then, we say $(X,D)$ is K-semistable if $\beta_{(X,D)}(E) \ge 0$ for every $E$.  

        Let $c \in \bQ^{>0}$.  Let $D = cQ$, where $Q$ is the hyperplane section at infinity of the cone $\bP(1,1,2)$.  Compute $\beta_{(X,cQ)}(E)$, where $E$ is the exceptional divisor of the resolution, and compute $\beta_{(X,cQ)}(Q)$.  Show that $(X,cQ)$ is K-unstable for all $c \ne 1/2$. 

    \end{enumerate}

\end{enumerate}

\section{Abban-Zhuang theory of admissible flags}\label{s:AZ}

While the invariants $\beta$ and $\delta$ are very useful for determining if Fano varieties are K-unstable, it is still very difficult to prove something is actually K-stable because you must check an inequality for every divisor $E$ over your variety.  One may ask how feasible it is to actually determine the stability of an arbitrary Fano variety.  In general, this is incredibly difficult, but one method for checking K-stability that has proved to be extraordinarily useful is the theory of admissible flags introduced by Abban and Zhuang in \cite{AZ}.  Very roughly, you may think about this as an adjunction result for K-stability: it allows you to restrict to smaller dimensional subvarieties and check appropriate inequalities there.  This leads to an inductive approach to determine K-stability.

To use this theory, we first need to introduce the language of filtrations.   We will do this using the language of log Fano pairs, but you may also set $\Delta = 0$ in what follows.  For more details and justification of the following results, see \cite{AZ}.

\begin{definition}
    If $L$ is a big line bundle on a variety $X$, the \textit{graded linear series} $\mathbf{V}_{\bullet} = \{ V_m\}_{m \in \mathbb{N}}$ to be $V_m = H^0(X, mL)$ for $m \in \mathbb{N}$ is called the \textbf{complete linear series associated to} $L$. 

    The \textbf{volume} of $\mathbf{V}_{\bullet}$ is $\vol(\mathbf{V}_\bullet) = \lim_{m \to \infty} \dim V_m / (m^n/n!) = \vol(L)$.
\end{definition}

We will use $L = -K_X - \Delta$ in what follows.  Now, we refine this series by a divisor $E$ over $X$.

\begin{definition}\label{def:refinement}
    For any divisor $E$ over $X$ and positive real number $t$, define the linear series \[(\calF_E V_m)_t = \{s \in V_m \mid \ord_E(s) \ge mt \},\] where we pullback $s \in V_m$ to a variety $Y$ extracting $E$. 
    
 For a real number $t$,
\[ \vol (\calF_E \mathbf{V}_\bullet )_t = \lim_{m\to\infty} \dim (\calF_E V_m)_t / (m^n/n!) .\]
Then, let \[ S(\mathbf{V}_\bullet; E) = \frac{1}{\vol \mathbf{V}_\bullet} \int_0^\infty \vol (\calF_E \mathbf{V}_\bullet)_t dt .\] 
\end{definition}

If $(X,\Delta)$ is a log Fano pair of $\dim X = n$ and $L = -K_X - \Delta$, perhaps convince yourself that this is just the definition of $S_{(X,\Delta)}(E)$.

\begin{definition}
    As in the previous section, we define the $\delta$-invariant \[\delta(X,\Delta;L) := \delta(X,\Delta;\mathbf{V}_\bullet) = \inf_{E} \frac{A_{(X,\Delta)}(E)}{S(\mathbf{V}_\bullet;E)}.\]
\end{definition}

\begin{definition}
        Let $Z$ be a subvariety of $X$.  We define 
        \[\delta_Z(X,\Delta;\mathbf{V}_\bullet) = \inf_{E: Z \subset C_X(E)} \frac{A_{(X,\Delta)}(E)}{S(\mathbf{V}_\bullet;E)}.\]
\end{definition}

It is clear that $\delta(X,\Delta;\mathbf{V}_\bullet) = \inf_{Z \subset X} \delta_Z(X,\Delta;\mathbf{V}_\bullet)$.  Furthermore, $(X, \Delta)$ is K-semistable if $\delta_p(X,\Delta; \mathbf{V}_\bullet) \ge 1$ for all points $p \in X$.  

Next, we `restrict' to $E$:

\begin{definition}
    Define the muligraded linear series \[W_{m,j}^E = Im (H^0(Y, mL-jE)) \to H^0(E, mL|_E - jE|_E))\] where $L|_E = -K_E - \Delta_E$ and $\Delta_E$ is the different, and $E|_E$ is a sensible divisor as long as $E$ is `nice' (precisely, we need $E$ to be of plt type).  

Then, define \[\vol(\mathbf{W}^E_{\bullet \bullet}) = \lim_{m \to \infty} \sum_{j\ge 0} \dim W_{m,j}^E / (m^n/n!) .\]  It is a theorem that in this set-up, $\vol(\mathbf{W}^E_{\bullet \bullet}) = \vol(\mathbf{V}_\bullet)$.
\end{definition}

Now, if we refine this multigraded linear series by a divisor $F$ over $E$, we analogously define the volume as in Definition \ref{def:refinement}.

\begin{definition}
    For a divisor $F$ over $E$ and a positive number $t$, define \[(\calF_F W_{m,j})_t = \{ s\in W_{m,j}^E \mid \ord_D (s) \ge mt \}\] and define 
\[ \vol (\calF_F \mathbf{W}^E_{\bullet \bullet} )_t = \lim_{m\to\infty} \sum_{j \ge 0} \dim (\calF_F W_{m,j}^E)_t / (m^{n}/{n}!) .\]

Finally, define \[ S(\mathbf{W}^E_{\bullet \bullet};F) = \frac{1}{\vol \mathbf{W}^E_{\bullet \bullet}} \int_0^\infty \vol (\calF_F \mathbf{W}^E_{\bullet \bullet})_t dt .\]
\end{definition}

Finally, we get to the adjunction-like result.  

\begin{theorem}
    For a primitive divisor $E$ over $X$ and any $Z \subset X$ such that $Z \subset C_X(E)$, let $\pi: Y \to X$ be a prime blow-up extracting $E$.  Then,

\[\delta_Z(X,\Delta; \mathbf{V}_\bullet) \ge \min \left\{ \frac{A_{(X,\Delta)}(E)}{S(\mathbf{V}_\bullet; E)}, \inf_{Z'} \delta_{Z'}(E, \Delta_E; \mathbf{W}^E_{\bullet\bullet}) \right\} \]
where the second infimum is taken over all $Z' \subset Y$ such that $\pi(Z')  = Z$, and 
\[ \delta_{Z'}(E, \Delta_E; \mathbf{W}^E_{\bullet\bullet}) = \inf_{F} \frac{A_{(E, \Delta_E)}(F)}{S(\mathbf{W}^E_{\bullet\bullet}; F)}\]
where $F$ is a prime divisor over $E$ with $Z' \subset C_E(F)$.  
\end{theorem}

This says that the $\delta$-invariant of $X$ can be bounded below by just the $\delta$-value computed by $E$ and then the $\delta$-invariant of a smaller dimensional variety ($E$).  Furthermore, it can be applied repeatedly to get an inductive result, reducing ultimately to the case $Z'$ is a point in a curve $E$.

\begin{remark}
    This method has been incredibly useful in the quest to determine K-(poly/semi) stability of every smooth Fano threefold.  For example, this is readily employed in \cite{Fano} to determine the K-(poly/semi)stability of the general member of every deformation type of smooth Fano threefolds.  In many cases, it is used to further show that \textit{every} member (not just general ones) are K-(poly/semi)stable. 
\end{remark}

We will use this to show several varieties are K-semistable, but first will introduce a more formulaic version due to \cite{Fano}. 

Suppose $(X,\Delta)$ is a klt pair with $\Delta$ effective.  Suppose $E$ is a prime divisor over $X$ of \textit{plt type}, i.e. there exists a morphism $\pi: \tX \to X$ extracting $E$ such that $-E$ is $\pi$-ample and $(\tX, \Delta + E)$ is plt, where $\tDelta$ is the divisor satisfying
\[ K_\tX + \tDelta + (1-A_{X,\Delta }(E))(E) = \pi^*(K_X + \Delta). \]

Define $\Delta_Y$ by 
\[ K_E + \Delta_E = (K_\tX + \tDelta + E)|_E .\]

Then, as above, 

\[\delta_Z(X,\Delta; \mathbf{V}_\bullet) \ge \min \left\{ \frac{A_{(X,\Delta)}(E)}{S(\mathbf{V}_\bullet; E)}, \inf_{Z'} \delta_{Z'}(E, \Delta_E; \mathbf{W}^E_{\bullet\bullet}) \right\} \]

where $Z'$ runs over subvarieties of $E$.  We can compute the first term: it is just $A_{X,\Delta}(E)/S_{X,\Delta}(E)$.  Furthermore, if $X$ is a surface, then $E$ is a curve, and must be smooth by the plt-type assumption.  Therefore, $Z'$ just ranges over points $p \in E$ with $\pi(p) = Z$, and if $p \in C_E(F)$, then $p = F$.  So, we simply need
\[ \delta_p(E, \Delta_E; \mathbf{W}^E_{\bullet\bullet}) = \frac{A_{E,\Delta_E}(p)}{S(\mathbf{W}^E_{\bullet\bullet}; p)}.  \]

The numerator is just $1 - \ord_p(\Delta_E)$, and the denominator is computed as follows.  Let $\tau$ be the pseudoeffective threshold of $\pi^*(-K_X -\Delta) - uE$ (the maximal $u$ such that this is pseudoeffective), and let $P(u) = P(\pi^*(-K_X -\Delta) - uE)$ by the positive part of the Zariski decomposition of $\pi^*(-K_X -\Delta) - uE$ and $N(u) = N(\pi^*(-K_X -\Delta) - uE)$ the negative part.  (You would have already had to find these to compute $S_{X,\Delta}(E)$.)  Provided that $p$ is not contained in the support of $N(u)$, then 
\[ S(\mathbf{W}^E_{\bullet\bullet}; p) = \frac{2}{\vol L} \int_0^\tau \int_0^\infty \vol(P(u)|_E - vp ) dv du. \]

If $t(u)$ is the pseudoeffective threshold of $P(u)|_E - vp$, this is just 
\[ S(\mathbf{W}^E_{\bullet\bullet}; p) = \frac{2}{\vol L} \int_0^\tau \int_0^{t(u)} \max \{ (\ord_p(P(u)|_E) - v ), 0 \} dv du. \]

In general, if $p \in N(u)$, we must add an extra term. This computation (for surfaces) gives the following result.  See \cite[Theorem 1.106]{Fano} for higher dimensional versions. 

\begin{theorem}{\cite[Theorem 1.106]{Fano}}
    In the above situation, let $P(u) = P(\pi^*(-K_X -\Delta) - uE)$ by the positive part of the Zariski decomposition of $\pi^*(-K_X -\Delta) - uE$ and $N(u) = N(\pi^*(-K_X -\Delta) - uE)$ the negative part.  Then, for any $p \in E$, 
    \[ A_{E,\Delta_E}(p) = 1 - \ord_p(\Delta_E)\] and
    \[ S(\mathbf{W}^E_{\bullet\bullet}; p) = \frac{2}{\vol L}  \int_0^\tau \left((P(u) \cdot E) \ \ord_p(N(u)|_E) +\int_0^{t(u)} \max \{ (\ord_p(P(u)|_E) - v ), 0 \} dv \right)du. \]
\end{theorem}

\begin{example}
Let us use this theory to show every smooth cubic surface is K-semistable.  Let $p \in X$ be a point in $X$.  We will bound $\delta_p(X)$. Let $E \subset X$ be an anticanonical divisor $E \in |-K_X|$ through $p$ (this exists as $-K_X$ is very ample).  

Because $\Delta = 0$ and $E$ is a curve on $X$, $A_{X}(E) = 1$.  We can also compute $S_{X}(E)$:
\begin{align*}
    S_{X}(E) &= \frac{1}{\vol(-K_X)} \int_0^\infty \vol(-K_X - tE) dt \\
    &= \frac{1}{3} \int_0^1 (-K_X + tK_X)^2 dt\\
    &= \frac{1}{3} \int_0^1 (1-t)^2(-K_X)^2 dt \\
    &= \frac{1}{3} \int_0^1 3(1-t)^2 dt \\
    &= \frac{1}{3} 
\end{align*}
Therefore, $A_{X}(E)/S_{X}(E) = 3 > 1$.  Now, as $-K_X - uE$ is nef if and only if it is pseudoeffective if and only if $0 < u < 1$, $P(u) = -K_X - uE$ and $N(u) = 0$.  Restricting to $E$, $P(u)|_E = (1-u)E|_E = 3(1-u)p$.  Then, $\ord_p(P(u)|_E) = 3(1-u)$, so  
\begin{align*}
    S(\mathbf{W}^E_{\bullet\bullet}; p) &= \frac{2}{\vol L} \int_0^\tau \int_0^{t(u)} \max \{ (\ord_p(P(u)|_E) - v ), 0 \} dv du \\
    &= \frac{2}{3} \int_0^1 \int_0^{3(1-u)} (3(1-u) - v) dv du \\
    &= 1.
\end{align*}

Therefore, \[ \delta_p(E, \Delta_E; \mathbf{W}^E_{\bullet\bullet}) = \frac{A_{E,\Delta_E}(p)}{S(\mathbf{W}^E_{\bullet\bullet}; p)} = 1. \]
So, for any $p \in X$, 
\[\delta_p(X;-K_X) = \delta_p(X; \mathbf{V}_\bullet) \ge \min \left\{ \frac{A_{(X,\Delta)}(E)}{S(\mathbf{V}_\bullet; E)}, \delta_{p}(E, \Delta_E; \mathbf{W}^E_{\bullet\bullet}) \right\} = \min \{ 3, 1 \} = 1. \]
Because $\delta(X;-K_X) = \inf_{p\in X} \delta_p(X;-K_X)$, we have proven $\delta(X; -K_X) \ge 1$ so $X$ is K-semistable.

In fact, one can prove these are actually K-stable using the last sentence of \cite[Thm. 1.2]{AZ}.
\end{example}

\begin{example}
   Next, we use this to show that the pair $(\bP(1,1,2), \frac{1}{2} Q)$ from the previous exercises is K-semistable.  For $p$ a smooth point, take $E$ to be a ruling through the point $p$.  In this case, $Z'$ will just equal $p$ as in the previous example.  For $p$ the singular point, take $E$ to be the exceptional divisor of the blow-up.  In this case, the $Z'$ will have to range through points on the curve $E$. 

   Suppose first that $p$ is a smooth point of $X = \bP(1,1,2)$ and let $E \subset X$ be a ruling through $p$.  Because $E$ is not contained in $\Delta = \frac{1}{2}Q$, $A_{X, \Delta}(E) = 1$.  We can also compute $S_{X, \Delta}(E)$:

\begin{align*}
    S_{X,\Delta}(E) &= \frac{1}{\vol(-K_X-\Delta)} \int_0^\infty \vol(-K_X - \Delta - tE) dt \\
    &= \frac{2}{9} \int_0^3 (3-t)^2E^2 dt\\
    &= \frac{1}{9} \int_0^3 (3-t)^2 dt \\
    &= 1
\end{align*}
Therefore, $A_{X}(E)/S_{X}(E) = 1$.

Now, as $-K_X -\Delta  - uE$ is nef if and only if it is pseudoeffective if and only if $0 < u < 3$, $P(u) = -K_X - \Delta - uE$ and $N(u) = 0$.  Restricting to $E$, $P(u)|_E = (3-u)E|_E = \frac{(3-u)}{2}p$.  Then, $\ord_p(P(u)|_E) = \frac{3-u}{2}$, so  

\begin{align*}
    S(\mathbf{W}^E_{\bullet\bullet}; p) &= \frac{2}{\vol L} \int_0^\tau \int_0^{t(u)} \max \{ (\ord_p(P(u)|_E) - v ), 0 \} dv du \\
    &= \frac{4}{9} \int_0^3 \int_0^{(3-u)/2} (\frac{(3-u)}{2} - v) dv du \\
    &= \frac{1}{2}.
\end{align*}

If $p \in \Supp \Delta_E$, then $A_{E,\Delta_E}(p) = \frac{1}{2}$, and otherwise $ = 1$.  Hence, \[ \delta_p(E, \Delta_E; \mathbf{W}^E_{\bullet\bullet}) = \frac{A_{E,\Delta_E}(p)}{S(\mathbf{W}^E_{\bullet\bullet}; p)} \ge 1. \]

For any $p$ other than the singular point, this proves that $\delta_p(X,\Delta; -K_X - \Delta) \ge 1$.  To complete the proof, it suffices to show this inequality for the singular point $p \in X$.  This is left to the exercises.
\end{example}

\subsection{Exercises}

\begin{enumerate}
    \item Finish Example 3.2.11.
\end{enumerate}

\section{Results on moduli of K-semistable Fano varieties}\label{s:Kmoduli}

Now that we understand how to show Fano varieties are (or are not) K-semistable, we will connect the ideas of K-stability with moduli of Fano varieties.  We will also continue to develop tools to understand K-stability to identify members of K-moduli spaces.

We first enumerate several results that make K-stability a good notion for moduli.  First, a discussion of moduli of Fano varieties in general and some examples of things we `want' from a moduli space. 

\begin{definition}
    A family of varieties $\calX \to T$ is \textbf{$\bQ$-Gorenstein} if $K_{\calX/T}$ is $\bQ$-Cartier.  
\end{definition}

This is a condition we usually impose on moduli problems because it makes things nicely behaved and yields a good moduli theory.  In fact, for technical reasons, we typically assume the \textit{Koll\'ar condition} that every reflexive power of $K_{\calX/T}$ commutes with base change.  

\begin{example}
In a $\bQ$-Gorenstein family, $K_{\calX/T}\vert_{\calX_t} = K_{\calX_t}$ is $\bQ$-Cartier, so ampleness of $-K_{\calX_t}$ is an open condition.  
\end{example}

\begin{question}
    We generally want our moduli spaces to be proper; i.e. ``limits exist in our moduli problem.''  In moduli of varieties of general type when $K_X$ is ample, we do this by taking $\Proj$ of some canonical section ring $R(K_X)$, which is finitely generated by \cite{BCHM}.  For Fano varieties, we instead know that $-K_X$ is ample, and do not have all of the nice results of the MMP at our disposal.  How can we construct limits of families of Fano varieties in a functorial way?
\end{question}

\begin{example}
    We also usually want our moduli spaces to be separated; i.e. ``families have unique limits.''  Here, we encounter a problem: the moduli space of mildly singular Fano varieties with fixed volume and dimension is not separated.  (Compare to: moduli of varieties with ample canonical divisor, where it is separated.) For example, let $X = \bP^1 $.  Then, $X$ is Fano and isotrivially degenerates to $X_0  = \bP^1 \cup \bP^1$ where the two curves are glued at one point; i.e. we can take a family of smooth conics (which are all isomorphic) degenerating to $xy = 0$.  The normalization of $X_0$ is $(\bP^1, \Delta) \cup (\bP^1, \Delta)$ where $\Delta$ is the conductor; one point on each $\bP^1$.  Then, $\vol(-K_X) = 2$; $\vol(-K_{X_0}) = 2\vol(-K_{\bP^1} - \Delta) =2$.  
\end{example}

\begin{example}\label{ex:manetti}
    To actually construct projective moduli spaces of varieties, we must: (1) bound our moduli problem in some way so that we can embed all of the varieties in question into a fixed projective space.  Then, (2) use the Hilbert scheme from that projective space to construct the moduli space (because Hilbert schemes are `nice').  Here, we encounter another problem: the set of log terminal Fano varieties with fixed volume and dimension is not necessarily bounded.  

    For example, we can construct an unbounded number of log terminal degenerations of $\bP^2$, all of which have anticanonical volume $9$.  Let $(a,b,c)$ be a solution to the Markov equation \[ a^2 + b^2 + c^2 = 3abc \] where $a,b,c$ are relatively co-prime, and consider the weighted projective space $\bP(a^2,b^2,c^2)$.  All solutions to the Markov equation are obtained by successively permuting or performing the \textit{mutation} $(a,b,c) \mapsto (a,b,3ab - c)$ starting from the minimal solution $(1,1,1)$. The first few triples in the Markov tree are
    \begin{center}
    \scalebox{1.2}
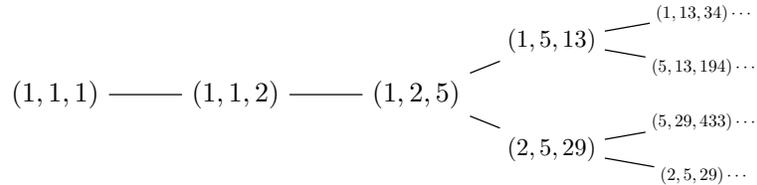
{\begin{tikzpicture}[branch1/.style ={scale=.4}]
\node[scale=.8] at (0,0) (1){$(1,1,1)$};
\node[scale=.8] at (2,0) (2){$(1,1,2)$};
\node[scale=.8] at (4,0) (3){$(1,2,5)$};
\node[scale=.7] at (5.5,.6) (31){$(1,5,13)$};
\node[scale=.5] at (7.2,.9) (311){$(1,13,34)\cdots$};
\node[scale=.5] at (7.2,.3) (312){$(5,13,194)\cdots$};
\node[scale=.7] at (5.5,-.6) (32){$(2,5,29)$};
\node[scale=.5] at (7.2,-.3) (321){$(5,29,433)\cdots$};
\node[scale=.5] at (7.2,-.9) (322){$(2,5,29)\cdots$};
\draw (1)--(2);
\draw (2)--(3);
\draw (3)--(31);
\draw (31)--(311);
\draw (31)--(312);
\draw (3)--(32);
\draw (32)--(321);
\draw (32)--(322);
\end{tikzpicture}}
    \end{center}
    corresponding to the weighted projective spaces $\PP^2$, $\PP(1,1,4)$, $\PP(1,4,25)$, $\dots$.  Any well-formed weighted projective space with anticanonical volume 9 corresponds to one of these (see exercises).  Furthermore, they all admit a smoothing to $\bP^2$ (see exercises).  There are infinitely many of these surfaces, they all have volume 9, have log terminal singularities, and are specializations of the fixed surface $\bP^2$.  (So, the moduli space here is unbounded and ``infinitely'' non-separated.)  
\end{example}

\begin{example}
    Finally, to get a well-behaved moduli space (a ``good quotient'' of the associated moduli stack), we also typically like to have that the automorphism groups of the elements parameterized by the moduli problem are at the very least reductive (even better: finite).  In any case, we know this is not true for Fano varieties (reductivity or finiteness).  Again, contrast with what happens for varieties with ample canonical divisor.
\end{example}

The moral of the previous set of `examples' is that it is probably hopeless to have a well-behaved moduli space of all Fano varieties analogous to that of the canonically polarized case.  But, all of these problems are \textit{solved} by restricting to only K-(semi/poly)stable Fano varieties.

\begin{theorem}[\cite{Jiang,Bir19}]
    The set of K-semistable Fano varieties of dimension $n$ and volume $V$ form a bounded family. 
\end{theorem}

\begin{theorem}[\cite{BLXOpenness,XuOpenness}]
    K-semistability is an open condition in $\bQ$-Gorenstein families. 
\end{theorem}

\begin{theorem}[\cite{LWX19,BX19,BHLLX,LXZ, XZ}]
    The moduli stack of K-semistable Fano varieties of dimension $n$ and volume $V$ is proper, and the moduli space of K-polystable Fano varieties of dimension $n$ and volume $V$ is projective. 
\end{theorem}

\begin{theorem}[\cite{ABHLX}]
    The automorphism group of a K-polystable Fano variety is reductive. 
\end{theorem}

This culminates in the K-moduli theorem, which we will not prove here: 

\begin{theorem}
    There is an Artin stack of finite type $\calM_{n,V}$ parameterizing families of K-semistable Fano varieties of dimension $n$ and volume $V$ and an associated projective good moduli space $M_{n,V}$ parameterizing K-polystable Fano varieties.  
\end{theorem}

In terms of the explicit issues raised above, we avoid the problem of degenerating to non-normal varieties (K-semistable implies log terminal, which implies normal) and the unbounded issue (other than $\bP^2$ itself, all of the weighted projective spaces gives in the unbounded example are K-unstable), and have properness and reductive automorphism groups.  

\begin{remark}
Everything we have said so far can be done for log Fano pairs $(X,D)$ by  replacing $-K_X$ with $-(K_X + D)$ in all of the definitions.  
\end{remark}

\subsection{Exercises}

\begin{enumerate}

    \item   Let $X=\bP(p,q,r)$ be a weighted projective space (with $p,q,r$ relatively co-prime) such that $(-K_X)^2 = 9$.  Prove that $p = a^2$, $q = b^2$, and $r = c^2$ such that $a^2 + b^2 + c^2 = 3abc$. 

    \item 

    (a) Show that the singularities on $\bP(a^2,b^2,c^2)$ where $a^2 + b^2 + c^2 = 3abc$ are $\bQ$-Gorenstein smoothable (i.e. locally around each singularity, construct a smoothing). 
    
    (b) Show that there are no local-to-global obstructions to deforming $\bP(a^2,b^2,c^2)$, so (a) together with the fact that $(-K_X)^2$ is constant in a $\bQ$-Gorenstein family implies that $\bP(a^2,b^2,c^2)$ is smoothable to $\bP^2$.  

    \item Prove that the general cubic surface is K-semistable using openness of K-semistability.  (Hint: find one with many automorphisms, like the Fermat or $xyz = w^3$, and compute $\alpha_G$.)
    \item Find a smooth Fano threefold that is K-semistable but not K-polystable. (Hint: find an isotrivial degeneration of a smooth Fano threefold to a K-polystable threefold.)
\end{enumerate}

\section{K-stability of singular Fano varieties and K-moduli of cubic surfaces}\label{s:cubics}

The goal of this section is to introduce more invariants related to the study of K-stability to completely determine several K-moduli spaces.  We have already learned some explicit tools for ``what K-stability is'' and that a K-moduli space exists.  But, how can we determine all of the objects in a particular K-moduli space?  We will introduce one more powerful invariant that is particularly useful in this setting.

\subsection{Local-to-global principles and normalized volume}

We start with a motivational theorem:

\begin{theorem}[\cite{Fuj18,Liu18}]\label{thm:volume1}
Assume $X$ is a K-semistable $\bQ$ Fano variety of dimension $n$.  Then, 
\[ 
(-K_X)^n \le (n+1)^n.
\]
Furthermore, equality holds if and only if $X \cong \bP^n$.  
\end{theorem}

\begin{proof}
We prove only the first statement.  Choose a smooth point $x \in X$ and let $Y = \mathrm{Bl}_x X$ be the blow up of the point $x$, with birational morphism $\mu: Y \to X$ and exceptional divisor $E \subset Y$.  

By assumption and the valuative criteria (Theorem \ref{thm:valcriteria}), we must have $\beta(E) \ge 0$, i.e. 
\[ 
A_X(E)(-K_X)^n \ge \int_0^\infty \vol (-K_X-tE) dt .
\]
Furthermore, denote $\mu: Y \to X$ the blow up  of a smooth subvariety $Z \subset X$ of codimension $k$ contained in the smooth locus of $X$ with exceptional divisor $E$. Then, \[ K_Y = \mu^*(K_X) + (k-1)E. \]
So, if we blow up a smooth point on a variety of dimension $n$, 
\[ K_Y = \mu^*(K_X) + (n-1)E \] and therefore
\[ 
A_X(E) = 1 + \mathrm{coeff}_E(K_Y - \mu^*K_X) = 1 + n -1 = n.
\]
To compute $\beta$, we can estimate $\vol(-K_X - tE) := \vol(\mu^*(-K_X) - tE)$.  Assume for simplicity that $t \in \bQ_{\ge 0}$.  Take an integer $m \in \bZ_{\ge 0}$ such that $mt \in \bZ_{\ge 0}$.  Then, by definition,
\[ \vol (\mu^*(-K_X) - tE) = \lim_{m \to \infty} \dfrac{h^0(Y,\calO_Y(m\mu^*(-K_X) - mtE))}{m^n/n!}.
\]
We can estimate the number of global sections: 
\[ \mu_* \calO_Y(m\mu^*(-K_X)-mtE) = \calO_X(-mK_X) \cdot \mathfrak{a}_{mt} 
\]
where $\mathfrak{a}_{mt} := m_x^{mt} = \{ f \in \calO_{x,X} \mid \ord_E(f) \ge mt \}$ is the (power of the) maximal ideal of the point $x$ we blew up.  Therefore, 
\begin{align*}
    h^0(Y,\calO_Y(m\mu^*(-K_X) - mtE)) &= h^0(X, \mu_*\calO_Y(m\mu^*(-K_X) - mtE)) \\
    &= h^0(X,\calO_X(-mK_X) \cdot \mathfrak{a}_{mt} ) \\
    &\ge h^0(X, \calO_X(-mK_X)) - \mathrm{length}(\calO_{x,X}/\mathfrak{a}_{mt}).
\end{align*}
The last inequality comes from the exact sequence 
\[ 0 \to \mathfrak{a}_{mt} \to \calO_X \to \calO_X/{\mathfrak{a}_{mt}} \to 0 \]
twisted by $\calO_X(-mK_X)$ (which is locally free in a neighborhood of $x$, so isomorphic to $\calO_X$ in a neighborhood of $x$, so twisting the third term in the sequence does nothing):
\[ 0 \to \calO_X(-mK_X)\cdot \mathfrak{a}_{mt} \to \calO_X(-mK_X) \to \calO_X/{\mathfrak{a}_{mt}} \to 0 .\]
The dimension of the global sections of the first sheaf is therefore bounded by the difference of the next two.  

This implies that 
\begin{align*}
    \vol (-K_X - tE) &\ge \vol(-K_X) - \lim_{m \to \infty} \dfrac{\mathrm{length}(\calO_{x,X}/\mathfrak{a}_{mt})}{m^n/n!} \\
    &= (-K_X)^n - \lim_{m \to \infty} \dfrac{\mathrm{length}(\calO_{x,X}/\mathfrak{a}_{mt})}{m^n/n!} \\
    &= (-K_X)^n - \lim_{mt\to \infty} \dfrac{\mathrm{length}(\calO_{x,X}/\mathfrak{a}_{mt})}{m^nt^n/n!}\cdot t^n \\
    &= (-K_X)^n - \vol(\ord_E) \cdot t^n \\
    &= (-K_X)^n - t^n.
\end{align*}

We will encounter volumes of valuations (the term $\vol(\ord_E)$ in the previous equation) momentarily, but you can also compute the last few lines as an exercise: let $k = mt$, and prove that $\lim_{mt \to \infty} \dfrac{\mathrm{length}(\calO_{x,X}/\mathfrak{a}_{mt})}{m^nt^n/n!} = 1$ (see the exercises).

Finally, plugging this into the inequality from the $\beta$-invariant (which we know holds if $X$ is K-semistable):
\[ 
A_X(E)(-K_X)^n \ge \int_0^\infty \vol(-K_X-tE) dt ,
\]
we find that 
\[ n (-K_X)^n \ge \int_0^\infty \mathrm{max}\{(-K_X)^n - t^n, 0\} dt\]
so 
\[ n(-K_X)^n \ge \dfrac{n}{n+1}(-K_X)^{n} \sqrt[n]{(-K_X)^n} \]
or 
\[ (-K_X)^n \le (n+1)^n.
\]
\end{proof}

We can strengthen this inequality with a tool called the \textit{normalized volume}.  This is defined in terms of valuations, but for a new learner of the subject, you can think about divisors whenever you see valuations: divisors correspond to so-called divisorial valuations by taking a divisor $E$ to the valuation $\ord_E$.  

\begin{definition}[\cite{ELS}]
Let $x \in X = \mathrm{Spec} R$ be a klt singularity and $v\in \mathrm{Val}_{x,X}$ be a valuation centered at $x$.  The \textbf{volume} of $v$ is 
    \[ \vol(v) = \lim_{k \to \infty} \frac{\mathrm{length}(R/\mathfrak{a}_k)}{k^n/n!} \]
    where $\mathfrak{a}_k = \{ f \mid v(f) \ge k \}$.
\end{definition}

There is also a definition of log discrepancy $A_X(v)$ for general valuations due to Jonsson and Musta\c{t}\u{a} \cite{JM} which we will not discuss in detail.  It coincides with the log discrepancy $A_{X}(E)$ for divisorial valuations.
 With these ingredients, we can define Li's normalized volume.

\begin{definition}[\cite{Li18}]
With the above set up, the \textbf{normalized volume} is 
\[ \widehat{\vol}(v) : = A_X(v)^n \cdot \vol (v) \]
and the \textbf{local volume} at $x$ is 
\[ \widehat{\vol}(x,X) : = \inf_{v\in \mathrm{Val}_{x,X}} \widehat{\vol}(v).\]
\end{definition}

\begin{remark} 
If $V$ is a $\bQ$-Fano variety, let $X = C(V,-rK_V)$ be the cone over $V$ and and $x\in X$ the vertex of the cone.  Because $V$ is Fano, $x\in X$ is klt, and $X$ has a partial resolution $\mu: Y \to X$ by blowing up the vertex with exceptional divisor $V_0 \cong V \subset Y$.  Another definition of K-semistability of $V$ is that \[ V_0 \text{ is a minimizer of } \widehat{\vol}(x,X) .\]  So, this notion of normalized volume also captures the stability. 
\end{remark} 

Using the normalized volume, we have a \textit{Local to Global} Theorem on the volume of K-semistable varieties.

\begin{theorem}[\cite{LL19}]\label{thm:localtoglobal}
Let $X$ be a K-semistable $\bQ$ Fano variety.  Then, for any $x \in X$, 
\[ (-K_X)^n \le \left( 1 + \frac{1}{n} \right)^n \widehat{\vol}(x,X).\]
\end{theorem}

\begin{proof}
The proof is the same as the proof of Theorem \ref{thm:volume1}, keeping $\vol(\ord_E)$ and $A_X(E)$ as in their definitions (without replacing them by $1$ and $n$).
\end{proof}

This is a very \textit{powerful} result, called a local-to-global theorem, because it relates the local invariants of the singularities (the normalized volume) to a global invariant (the anticanonical volume).  So, it allows you to constrain what singularities can appear relative to the anticanonical volume and vice versa.  We will use this to our advantage later.

Here are some properties of the local volume function $\widehat{\vol}(x,X)$: 

\begin{property}\label{propsofvol}
\begin{enumerate}
    \item \cite{dFEM,Li18} If $X$ has dimension $n$ and $x \in X$ is smooth, then 
    \[ \widehat{\vol}(x,X) = n^n \]
    \item \cite{LXcubics} If $X$ has dimension $n$, then for any $x \in X$, \[ \widehat{\vol}(x,X) \le n^n \] and equality holds if and only if $x$ is smooth.  Combining this with Theorem \ref{thm:localtoglobal} gives Theorem \ref{thm:volume1}.
    \item \cite{Liu18} If $x \in X = (0 \in \mathbb{A}^n/G)$ is a quotient singularity where $G \subset \GL_n(\bC)$ acts freely in codimension 1, then \[\widehat{\vol}(x,X) = \frac{n^n}{|G|}.\] If $x \in X$ is a quotient of an arbitrary variety by $G \subset \GL_n(\bC)$, then \[\widehat{\vol}(x,X) \le \frac{n^n}{|G|}.\] Combining this with Theorem \ref{thm:localtoglobal} gives: if $X$ is a K-semistable Fano variety, then for any quotient singularity $x \in X = (0 \in \mathbb{A}^n/G)$, 
    \[ (-K_X)^n \le \frac{(n+1)^n}{|G|}. \]
    \item \cite{LXcubics} If $X$ has dimension $n = 2 $ or $n = 3$ and $x\in X$ is \textbf{not} a smooth point, then 
    \[ \widehat{\vol}(x,X) \le 2(n-1)^n \] and equality holds if and only if $x$ is an ordinary double point.  Conjecturally, the ordinary double point \textit{always} gives the second largest volume. 
\end{enumerate}
\end{property}

\begin{definition}
    The singularity $\frac{1}{n}(a,b)$ is the surface quotient singularity obtained by the action $\mathbb{A}^2/\mu_n$, where a primitive root of unity $\zeta_n$ acts by $\zeta_n \cdot (x,y) = (\zeta_n^a x, \zeta_n^b y)$.  
\end{definition}

\begin{example}
Previously, you proved in the exercises that $\bP(1,1,2)$ is K-unstable using the $\beta$-invariant.  Let's prove it again using the normalized volume.  

By Property \ref{propsofvol}(3), we know that if $X$ is a K-semistable Fano variety with a quotient singularity $\mathbb{A}^n/|G|$, then $(-K_X)^n \le \frac{(n+1)^n}{|G|}$.  Let's plug in the associated values for $\bP(1,1,2)$: 
    \begin{itemize}
        \item $n = 2$ (the dimension of $X$)
        \item Because $X$ is a (singular) quadric surface in $\bP^3$, by adjunction, $(K_{\bP^3}+X)|_X = K_X$, so $\calO_X(K_X) = (\calO_{\bP^3}(-2))|_X$, and therefore \[ K_X^2 = (\calO_{\bP^3}(-2)|_X)^2 = \calO_{\bP^3}(-2) \cdot \calO_{\bP^3}(-2) \cdot \calO_{\bP^3}(2) = 8.   \]

        Alternatively, you could compute $(-K_X)^2$ using intersection theory on weighted projective space: $\calO(K_X) = \calO(-1-1-2) = \calO(-4)$, and $(-K_X)^2 = \frac{(-4)^2}{2} = 8$.

        \item The singularity on $\bP(1,1,2)$ can be described as the quotient singularity $\frac{1}{2}(1,1)$.  So, $|G| = |\mu_2| = 2$.
    \end{itemize}

Now we plug in!  We see that \[ (-K_X)^2 = 8 > \frac{3^2}{2} = \frac{9}{2} \] so $X$ is K-unstable. 
\end{example}

See the exercises for practice with quotient singularities.

\subsection{K-moduli of cubic surfaces}

Now, let's use the normalized volume to do some moduli!  In a moduli problem, we wish to classify the objects that can appear.  Suppose $X$ is a K-semistable object in some moduli space.  Theorem \ref{thm:localtoglobal} can often be used to give a bound on the index of $-K_X$, which is related to the singularities that can appear on $X$, and the various reformulations can give more precise statements.

    Consider a smooth cubic surface $X$ in $\bP^3$.  By adjunction, $\calO(K_X) = \calO_{\bP^3}(-1)|_X$, so $X$ is Fano and $(-K_X)^2 = 3$.  In other words, cubic surfaces are examples of degree three del Pezzo surfaces.

\begin{question}
    What does the moduli space $\calM^{\mathrm{sm}}_{2,3}$ of K-(semi/poly)stable degree three del Pezzo surfaces look like?  (Here, the superscript $\mathrm{sm}$ indicates that we are only looking at \textit{smoothable} surfaces.  In fact, there are no other components as was proven by this Bootcamp group, which will appear in forthcoming work.)
\end{question}

From the exercises, every smooth del Pezzo surface of degree 3 is a cubic surface.  What about the singular ones?  Suppose $X$ is a K-semistable singular del Pezzo surface, and let $x \in X$ be a singular point.  We know $X$ is normal and log terminal by the K-semistable assumption.  As, log terminal surface singularities are all quotient singularities, we can use Property \ref{propsofvol}(3) to bound the normalized volume.  Write $(x \in X) = (0 \in \mathbb{A}^2/G)$.  

Theorem \ref{thm:localtoglobal} together with Property \ref{propsofvol}(3) says \[(-K_X)^2 \le \frac{9}{|G|}. \]
We know $(-K_X)^2 = 3$, and we are assuming $x \in X$ is not smooth (so $|G| > 1$) so this implies that 
\[ 2 \le |G| \le 3 .\]

In other words, $|G| = 2$ or $|G| = 3$.  There are only three choices for the resulting singularity $x \in X$: 
    \begin{enumerate}
        \item $G = \mu_2$ and $x \in X$ is an $A_1$ (or $\frac{1}{2}(1,1)$) singularity, which is the quotient $\mathbb{A}^2/{\mu}_2$ where $\mu_2$ acts by $-1 \cdot (x,y) = (-x, -y)$
        \item $G = \mu_3$ and $x \in X$ is an $A_2$ (or $\frac{1}{3}(1,2)$) singularity, which is the quotient $\mathbb{A}^2/\mu_3$ where a cube root of unity $\zeta_3 \in \mu_3$ acts by $\zeta_3 \cdot (x,y) = (\zeta_3 x, \zeta_3^{-1} y)$
        \item $G = \mu_3$ and $x \in X$ is a $\frac{1}{3}(1,1)$ singularity, which is the quotient $\mathbb{A}^2/\mu_3$ where a cube root of unity $\zeta_3 \in \mu_3$ acts by $\zeta_3 \cdot (x,y) = (\zeta_3 x, \zeta_3 y)$
    \end{enumerate}

But, by the classification of smoothable log terminal surface singularities (e.g. \cite[\S 3]{KSB88} or \cite[\S 6.6]{Kol23}), the third choice in the list is \textbf{not} smoothable.  So, $x \in X$ must be an $A_1$ or $A_2$ singularity.  By Exercise \ref{ex:Gorenstein}, we know that $A_n$ singularities are Gorenstein, so any K-semistable del Pezzo surface of degree 3 $X$ is Gorenstein, so $-K_X$ is Cartier.  In fact, once we know it is Cartier, it is very ample by a result of Fujita (this is true for cubics in any dimension--see \cite{Fuj90}) so $|-K_X|: X \hookrightarrow \bP^3$ as a (singular) cubic surface.  

So far, we have shown: 

\begin{theorem}
    If $[X] \in \calM^{\mathrm{sm}}_{2,3}$ is a K-semistable $\bQ$ Fano surface of degree 3, then $X$ is a cubic surface in $\bP^3$ with at worst $A_1$ or $A_2$ singularities. 
\end{theorem}

Now, we know that any element parameterized by $\calM^{\mathrm{sm}}_{2,3}$ is really just a surface in $\bP^3$.  To determine the K-stability of such a surface, does that mean we are allowed to restrict to test configurations where the central fiber is also in $\bP^3$?  In other words, can we consider only one-parameter subgroups of $\PGL_4$ in the test configuration definition?  

Depending on your background, this \textit{might} be ringing some sort of bell.  If we have objects in $\bP^n$, and degenerate along one-parameter subgroups of $\PGL_{n+1}$, and compute some sort of weight of this action.... This looks just like GIT!   

\begin{theorem}[\cite{OSS}]\label{cubics}
    GIT = K stability for cubic surfaces.  A cubic surface $X$ is K-(poly/semi) stable if and only if it is GIT (poly/semi) stable.  It follows that the K-moduli space is isomorphic to the GIT moduli space.
\end{theorem}

\begin{proof} (Sketch.)
\textbf{Step 0: K stability $\implies$ GIT stability.}

First, we show K $\implies$ GIT (this is a general idea due to Paul and Tian \cite{PT06} for hypersurfaces).  \textit{Basic idea: one parameter subgroups are test configurations, so if all of the test configurations have positive weight, then so should all the one-parameter subgroups.}

By assumption, if $X$ is K-(semi)stable, we have $\Fut(\mathcal{X}, \mathcal{L}) (\ge) > 0 $ for any test configuration.  And, we proved it is a hypersurface $X \subset \bP^3$, so given any one-parameter subgroup $\lambda \subset \PGL_4$, this induces a test configuration $(\calX_\lambda, \calL_\lambda)$.  

Paul and Tian \cite{PT06} show that that the Futaki invariant is proportional to the GIT weight, i.e. \[ \Fut(\calX_\lambda, \calL_\lambda) = a \mu^{\mathcal{O}(1)}([X], \lambda) \]
where $a > 0$ is a positive constant and $\mu^{\mathcal{O}(1)}$ is the GIT weight.  Therefore, K-semistability of $X$ implies that the GIT weight is $\ge 0$ for every one-parameter subgroup, hence the Hilbert-Mumford criterion implies that $X$ is GIT-semistable. 

Now, we  want to show GIT stability $\implies$ K stability: 
Suppose $X \subset \bP^3$ is GIT polystable (the other cases are similar).
We want to show that $X$ is K-polystable.  

\textbf{Step 1: Openness of K-moduli.}  There exists a K-stable cubic surface (see the exercises, or use \S \ref{s:AZ}) and the K-stable locus is Zariski open, so the general one is K-stable.

\textbf{Step 2: Properness of K-moduli.} Take a smoothing $\mathcal{X} \to C$ over a pointed curve $0 \in C$ such that $\calX_0 \cong X$ is the cubic surface we assume to be GIT polystable.  The general fiber $\calX_t$ is a smooth cubic surface and, from the previous step, we can assume $\calX_t$ is K-stable.  By properness of K-moduli, up to base change, there exits a family $\calX' \to C$ such that $\calX' \setminus \calX'_0 \cong \calX \setminus \calX_0$ and $X' : = \calX'_0$ is K-polystable.  

\textbf{Step 3: Local to Global Volume Comparison.} 
From our work already using Theorem \ref{thm:localtoglobal}, because $\calX_0'$ is K-polystable, it is a cubic surface.  By Step 0, because K-polystability implies GIT-polystability, $\calX_0'$ is a GIT polystable surface.  But now, $X = \calX_0$ and $\calX_0'$ are two polystable limits of the same family of surfaces, so by separatedness of the GIT moduli space, we must have $X \cong \calX_0'$.  Therefore, $X$ is K-polystable.
\end{proof}

\begin{corollary}
    Because all smooth cubic surfaces are GIT stable, this implies that all smooth cubics are K stable. 
\end{corollary}

Using the index bound from the normalized volume, a similar result is true in higher dimensions: 

\begin{theorem}[\cite{LXcubics,LiuQuartics}]
    GIT = K stability for cubic threefolds and cubic fourfolds. 
\end{theorem}

This is expected to hold in higher dimensions, and would follow from the conjectural Property \ref{propsofvol}(4) in higher dimensions. 

\begin{conj}
GIT = K stability for cubic hypersurfaces. 
\end{conj}

\subsection{Exercises}

\begin{enumerate}
    \item Prove that there are no nontrivial K-semistable degenerations of $\bP^2$ and $\bP^1 \times \bP^1$ (show any degeneration must be smooth using the local volume, and use rigidity of smooth Fanos).  
    
    \item The singularities $\frac{1}{4}(1,1)$ and $\frac{1}{4}(1,3)$ are smoothable, so could appear on K-semistable degenerations of del Pezzo surfaces.  What is the maximal degree of a del Pezzo surface for which they could appear?  Bonus: exhibit a degree $d$ del Pezzo surface with at least one of these singularities.

    \item Prove that any weighted projective space $\bP(a_0, \dots, a_n)$ not equal to $\bP^n$ is K-unstable. 
    
    \item In the proof of the index bound, we used the following: 
    
    If $x \in X$ is a smooth point (you may assume $x \in X = 0 \in \mathbb{A}^n$), show that \[\lim_{k \to \infty} \dfrac{\mathrm{length}(\calO_{x,X}/\mathfrak{a}_{k})}{k^n/n!} = 1.\] 

    Prove this.

    \item If $(X,cD)$ is a K-semistable log Fano pair, then the index bound inequality is:  
        \[ (-K_X-cD)^n \le \left( 1 + \frac{1}{n} \right)^n \widehat{\vol}(x,X,cD).\]
    If $x \in X$ is a quotient singularity by a group $G$ and $x \notin D$, then it is still true that 
    \[ \widehat{\vol}(x,X,D) \le \frac{n^n}{|G|}.\]

    Suppose $X = \bP(1,1,4)$ and $D \in \calO_X(4d)$ for some integer $d$.  
        \begin{enumerate}
            \item If $D$ passes through the singular point of $X$, show that it must have multiplicity at least 4 at the singular point, and that $(X,cD)$ is K-unstable for any $c \in (0,\frac{3}{2d})$ by computing $\beta(E)$ where $E$ is the exceptional divisor of the blow up of the singular point.  
            \item If $D$ does not pass through the singular point, use the index bound to prove that $(X,cD)$ could only be K-semistable if $c \ge \frac{3}{4d}$. 
        \end{enumerate}

    \item Let $X$ be a degree $d$ smooth del Pezzo surface.   Prove that $-K_X$ is very ample and the linear system $|-K_X|$ embeds $X \hookrightarrow \bP^d$ as a degree $d$ surface.

    \item If $X$ is a Gorenstein surface with ample $-K_X$ such that $(-K_X)^2 = 3$, prove that $|-K_X|$ is base point free and therefore very ample so $X$ embeds in $\bP^3$ as a cubic surface.

    \item\label{ex:Gorenstein} Prove that an $A_n$ singularity, the quotient $\mathbb{A}^2/ \mu_n$ where $\mu_n$ acts by $\zeta_n\cdot (x,y) = (\zeta_n x, \zeta_n^{-1} y)$, is Gorenstein.  (Hint/fact: any hypersurface singularity is Gorenstein, and the quotient of $\mathrm{Spec} k[x_1, \dots, x_n ]$ by a finite group $G$ is $\mathrm{Spec} (k[x_1, \dots, x_n]^G$, the ring of invariant poynomials under the group action $G$.) 

    \item Prove that, in the minimal resolution of an $A_n$ singularity, the exceptional divisor is a chain of smooth rational curves each with self intersection $-2$.  

    \item Another quotient singularity is the quotient $\mathbb{A}^2/ \mu_n$ where $\mu_n$ acts by $\zeta_n\cdot (x,y) = (\zeta_n x, \zeta_n y)$.  
        \begin{enumerate}
            \item The rational normal curve of degree $n$ is defined as the image of the embedding of $\bP^1 \to \bP^{n}$ given by evaluation on the sections of $\mathcal{O}_{\bP^1}(n)$ (i.e. $[x:y] \mapsto [x^n:x^{n-1}y:\dots:y^n]$).  Prove that the cone over this curve has a singularity of type $\frac{1}{n}(1,1)$. 
            \item Prove that, for any $n$, the exceptional divisor of the minimal resolution of the cone over the rational normal curve of degree $n$ is a single rational curve with self-intersection $-n$, and compute the discrepancy of the exceptional divisor. 
            \item Prove that the cone over the rational normal curve of degree $n$ is isomorphic to the weighted projective space $\bP(1,1,n)$.  
            \item Prove that $\bP(1,1,n)$ is K-semistable if and only if $n = 1$. 
        \end{enumerate}
\end{enumerate}

\section{Concluding remarks}

We conclude this survey with a few remarks on explicit K-stability and K-moduli in general.  K-stability has proven to be an extraordinarily useful idea in constructing moduli spaces of Fano varieties, yet it is often explicitly difficult to determine.  We have seen already, for each degree $d$, whether or not a smooth degree $d$ del Pezzo surface is K-(semi)stable.  We also have a complete description of the K-moduli space of smoothable del Pezzo surfaces of each degree $d$ by \cite{OSS}.  

For Fano threefolds, the situation is much more difficult.  In \cite{Fano}, the authors determine whether or not the general member of each family of smooth Fano threefolds is K-semistable.  However, at this time we still do not know for each family precisely which smooth members are K-semistable (all? only some?).  We are also still far from understanding the entire K-moduli space of smoothable Fano threefolds, although this has been worked out in several explicit examples (see, for example \cite{ADL21} for quartic double solids, \cite{218} for conic bundles that are double covers of $\bP^1 \times \bP^2$ branched over a (2,2) divisor, and a few additional examples in \cite{LiuZhao2-15, qubits}).  

In higher dimensions or in generality, the situation does not improve.  A fundamental difficult open question in this area are: 

\begin{question}
    Is every smooth Fano hypersurface of degree $\ge 3$ in $\bP^n$ K-stable?
\end{question}

We at least know the generic such hypersurface is K-stable by finding just one K-stable surface (e.g. the Fermat) and using openness of K-stability.  One could ask the same question for complete intersections, but in this case we do not even know a generic complete intersection is K-stable. 

An alternative perspective on the subject is to use the theory of \textit{wall crossing} to understand K-moduli.  This is the approach taking in several papers, including \cite{ADL19, ADL20, ADL21} and used in \cite{218}.  While not mentioned here, the theory of wall crossing will be surveyed in a separate article for the 2025 Summer Research Institute conference volume.

\bibliographystyle{alpha}
\bibliography{moduliofvarieties}

\end{document}